\documentclass[12pt]{amsart} 
\usepackage{amsfonts,latexsym,wasysym, mathrsfs, mathtools,hhline,color, bm}
\usepackage[all, cmtip]{xy}

\usepackage{url}

\definecolor{hot}{RGB}{65,105,225}

\usepackage[pagebackref=true,colorlinks=true, linkcolor=hot ,  citecolor=hot, urlcolor=hot]{hyperref}

\usepackage{ textcomp }
\usepackage{ tipa }

\usepackage{graphicx,enumerate}

\usepackage{enumitem}
\usepackage[normalem]{ulem}

 \usepackage{amsthm}  
 \usepackage{amsmath} 
 \usepackage{amsfonts}  
 \usepackage{amssymb}
 \usepackage{amscd} 
 \usepackage[all]{xy} 

\usepackage[dvipsnames]{xcolor}  
\usepackage{graphicx, subfigure}
\usepackage{pdfsync}
\usepackage{url}
\usepackage{multirow}

\usepackage[linesnumbered,ruled,vlined]{algorithm2e}

\newtheorem{thm}{Theorem}[section]

\theoremstyle{definition}
\newtheorem{example}[thm]{Example}
\newtheorem{theorem}[thm]{Theorem}

\newtheorem{compresult}[thm]{Result}
\newtheorem{prop}[thm]{Proposition}

\theoremstyle{definition}
\newtheorem{defn}[thm]{Definition}
\newtheorem{rem}[thm]{Remark}
\numberwithin{equation}{section}

\newcommand{\Wpi}{W_\pi}

\newcommand{\bC}{{\bf{C}}}

\newcommand{\ba}{{\bf{a}}}
\newcommand{\bp}{{\bf{p}}}
\newcommand{\bb}{{\bf{b}}}
\newcommand{\bd}{{\bf{d}}}
\newcommand{\bK}{{\bf{K}}}
\newcommand{\bM}{{\bf{M}}}

\newcommand{\uFiber}{{\pi^{-1}(u^\star)}}

\newcommand{\cG}{{\mathcal G}}

\newcommand{\cC}{{\mathcal C}}
\newcommand{\cB}{{\mathcal B}}

\newcommand{\CC}{{\mathbb{C}}}

\newcommand{\Wpiq}{W_\pi(q^{\star})}
\newcommand{\Wpia}[1]{W_\pi(\alpha,#1)}
\newcommand{\Wpib}[1]{W_\pi(\beta,#1)}

\newcommand{\defcolor}[1]{{\color{NavyBlue}#1}} 
\newcommand{\demph}[1]{\defcolor{{\sl #1}}}

\newcommand{\upstairs}{\mathbb{C}^k\times\mathbb{C}^n}
\newcommand{\branchLocus}{\mathcal{B}}
\newcommand{\Gr}{G}

\DeclareMathOperator{\Tr}{Tr}
\newcommand{\Fbar}{\bar F}

\title{Solving parameterized polynomial systems with decomposable projections}

\author[Am\'endola]{Carlos Am\'endola}
\address{
Carlos Am\'endola \\
Technische Universit\"at M\"unchen, Germany
}
\email{carlos.amendola@tum.de}
\urladdr{}

\author[Lindberg]{Julia Lindberg}
\address{
Julia Lindberg \\
University of Wisconsin-Madison \\ USA
}
\email{jrlindberg@wisc.edu}
\urladdr{}

\author[Rodriguez]{Jose Israel Rodriguez}
\address{
Jose Israel Rodriguez\\
University of Wisconsin-Madison \\
 USA
} 
\email{jose@math.wisc.edu}
\urladdr{}

\begin{document}

\begin{abstract}
The  Galois  group  of  a  parameterized  polynomial  system  of equations  encodes  the  structure  of  the  solutions. This monodromy group  acts  on  the  set of solutions for a general set of parameters, that is, on the fiber  of  a projection  from  the  incidence  variety  of  parameters  and  solutions  onto  the  space  of  parameters.  When this projection is decomposable, the Galois group is imprimitive, and we show that the structure can be exploited for computational improvements. Furthermore, we develop a new algorithm for solving these systems based on a suitable trace test. We illustrate our method on examples in statistics, kinematics, and benchmark problems in computational algebra. In particular, we resolve a conjecture on the number of solutions of the moment system associated to a mixture of Gaussian distributions.
\end{abstract}

\maketitle

\section{Introduction}

A \demph{parameterized system of polynomial equations} $F=0$ arises from a polynomial map 
$F:\upstairs\to \CC^N$ where 
$\CC^k$ is the \demph{space of parameters}, 
$\CC^n$ is the \demph{space of solutions} and $N$ is the number of equations.

The polynomial map $F$
gives rise to the incidence variety:
\[
V(F) := \{
(u,z)\in \CC^k\times\CC^n :
F(u,z)=0
\}.
\]
 The projection of the incidence variety to $\CC^k$ has a fiber over a general point in the image. When this fiber is zero dimensional, its cardinality  gives a general root count to the system of equations. 

The \demph{Galois group} or monodromy group is an invariant of a general fiber, that is, an invariant of the solutions to the parameterized polynomial system. This group acts on the solutions by permuting the elements of the fiber.
When considering an irreducible component of the incidence variety,  the Galois group is known to be transitive.
With the transitivity property, one is able to use numerical homotopy continuation to collect solutions of the system if given a starting point. This powerful technique has been used in many instances \cite{Bertini,kileelchen,DelCampo,polsyshommon,tensorDecomp,dhagash,powerbus}. 
In addition, numerical algorithms for computing Galois groups have been developed in  \cite{galois,LS09Galois}, and examples from applications include formation shape control and maximum likelihood estimation in algebraic statistics. 

Many of these instances have a Galois group with special block symmetries, and we say the group is \emph{imprimitive} (see Definition \ref{def:imprimitive}).

The main theoretical connection is that the Galois group of a parameterized polynomial system is imprimitive if and only if the system has a decomposable projection (Proposition \ref{prop:structureDecomp}). 
We exploit this structure by generalizing witness sets of projections in Section \ref{s:prelim}, leading to Algorithm \ref{algo:segreTrace}.

An illustrative simple example is the following. 

\begin{example}\label{ex:first}
Let $k = n= N = 1$. Consider the curve $Z$ defined by $z^{2000}-2z^{1000}+u=0$. 
The projection map
\[
\begin{array}{c}
\pi:Z\to Y\\
\left(u,z\right)\mapsto u
\end{array}.
\]
of this curve to the $u$-coordinate is 2000 to one for all $u$ in $\mathbb{C}\setminus\{0,1\}$.
The Galois group of the
cover associated to $\pi$ is
not the full symmetric group $\mathfrak{S}_{2000}$. 
This is seen by decomposing 
the projection as the following sequence of maps:
\[
\begin{array}{c}
Z\stackrel{\alpha}{\rightarrow}V\stackrel{\beta}{\rightarrow}Y\\
\left(u,z\right)\mapsto\left(u,z^{1000}\right)\mapsto u.
\end{array}
\]
The degree of the map $\alpha:Z\to V$  is $1000$.
By setting $y=z^{1000}$, the defining equation of $V$ is seen to be $y^{2}-2y+u=0$. See Figure \ref{fig:exfirst}.
The map $\beta:V\to Y$ is a projection with degree
two. 
Thus, we have (non-trivially) decomposed the projection $\pi$ into a composition
of maps $\alpha \circ \beta$. 
One way to describe the the fiber $\pi^{-1}\left(u\right)$ is by listing all $2000$ points over a general point.
We prefer to list only $2=\deg\beta$ points that map to distinct points under $\alpha$.
Often, this description is sufficient as the other solutions are equivalent up to an easily described action. 
In this example, the action is given by multiplying the $z$-coordinate by a primitive $1000$th root of~unity. 
\end{example}

\begin{figure}[h]
\centering
\includegraphics[scale=0.3]{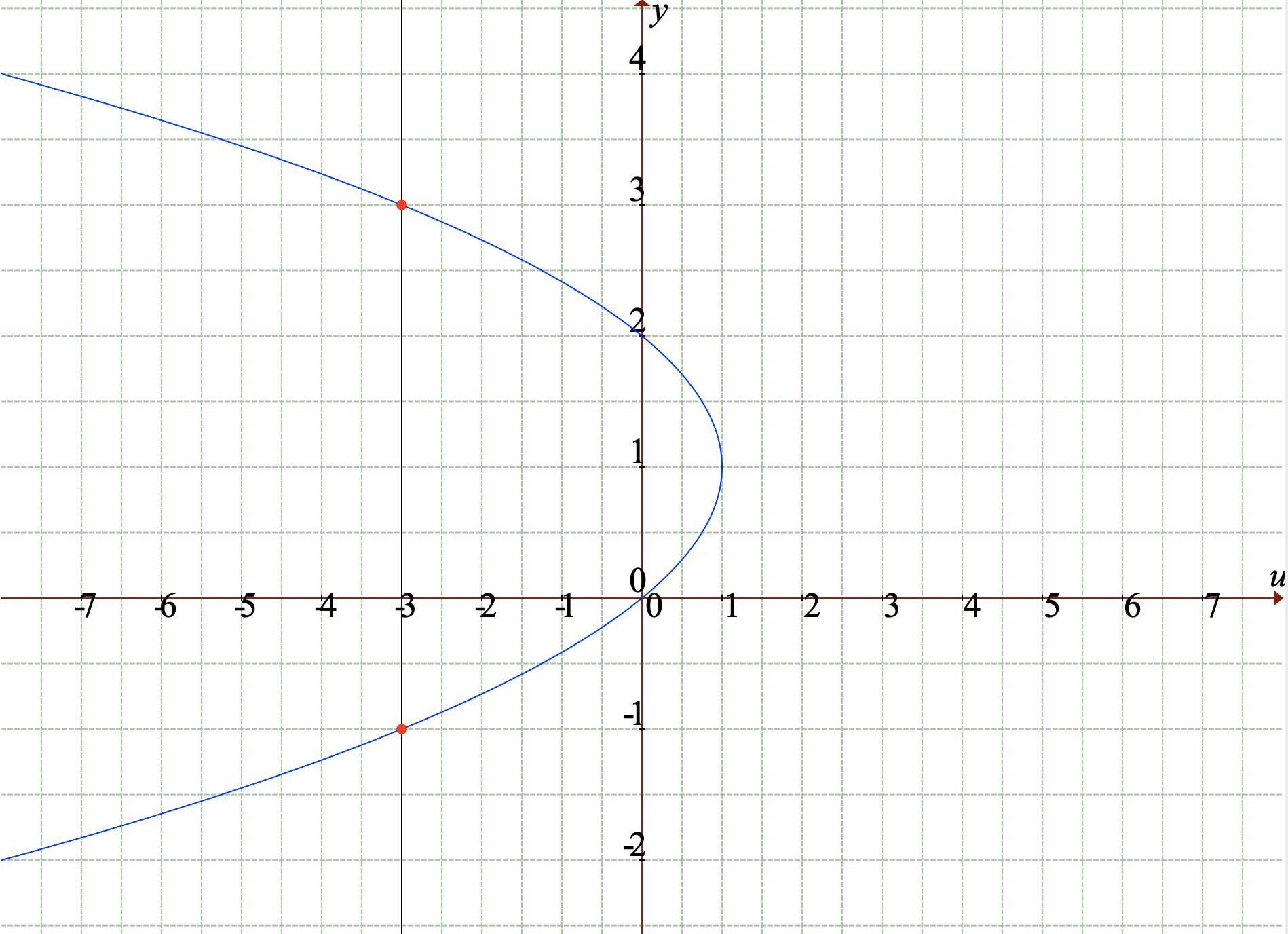}
\caption{Curve $y^2-2y+u=0$ from Example \ref{ex:first}. When $u=-3$ we obtain the $\deg \beta = 2$ intersection points $(-3,-1), (-3,3)$ of $V$.}

\label{fig:exfirst}
\end{figure}

In the  above,  eliminating the $z$-coordinate to compute the defining equation of $V$ is easily performed via substitution. 
However, in examples of Section \ref{sec:exShort}, 
this elimination is a bottleneck that  we avoid by using the numerical homotopy continuation method of monodromy.

The rest of the paper is structured as follows. 
In Section~\ref{ss:GaloisGroups}, we review monodromy groups and the notion for a group to be imprimitive. In Section \ref{ss:decomp} we define decomposability for a projection and explain its connection to imprimitivity of the Galois group. In Section \ref{sec:ws}, we define witness sets, coming from numerical algebraic geometry, in the decomposable context. We put things together in a decomposable monodromy algorithm in Section~\ref{ss:decmon}, and illustrate with elementary examples in Section~\ref{ss:examples}. In order to have a stopping criterion for the first algorithm, we introduce a trace test algorithm in Section~\ref{ss:tracetest}. Finally, Section~\ref{sec:exShort} is devoted to applications: we disprove 
a conjecture on Gaussian mixtures models,
explore a kinematics problem, and provide some computations related to the benchmark cyclic $n$-roots problem. 
 
\section{
Monodromy, decompositions, and invariants
}\label{s:prelim}

In this section, we recall basic facts about monodromy groups of parameterized polynomial systems, we define decomposability and witness sets.

\subsection{Monodromy and Galois groups}\label{ss:GaloisGroups}
\
As defined in the introduction, we will consider the incidence
variety ${V(F)}:= \{ (u,z)\in \CC^k\times \CC^n : F(u,z)=0\} $
of a square system of polynomial equations $F=0$ parameterized by $\CC^k$ and their sets of solutions in $\CC^n$.
We consider the fiber of the projection $\pi:V(F)\to\CC^k$.
We denote the fiber over a point $u^{\star}\in \pi(V(F))$ as $\uFiber:= \{ z\in \CC^n : F(u^{\star},z) = 0 \}$. 
In the cases we are interested in, we can assume this fiber is finite over a generic point in $\CC^k$.
In other words, we are assuming 
the dimension of ${V(F)}$ is $k$ and $\pi:V(F)\to\CC^k$  is dominant. 

As a consequence, we see that
$\pi:V(F)\to\CC^k$ 
is a branched cover, with a branch locus denoted by $\cB$.
Over  $\CC^k\setminus\cB$, the projection $\pi$ admits a covering space, which has a monodromy group.
This monodromy group is equivalent to the Galois group, see \cite{harris1979} for a modern reference.

\begin{defn}
Let $\gamma\subset\CC^k\setminus\cB$ denote a loop in $\CC^k$
based at $u^\ast$. 
Then, $\gamma$ induces an action on $\uFiber$. 
We denote the permutation of the fiber induced by $\gamma$ as $\sigma_{\gamma}$.
The group of such permutations is the monodromy group ${\cG_{\pi:V(F)\to\CC^k}}$.
When it is clear, we denote this group by ${\cG_{F}}$.
\end{defn}

\begin{rem}
When $V(F)$ is a curve $(k=1)$, the branch locus is a finite set of points.  
The monodromy group is generated by fixing a base point and taking simple loops around each of these branch points.
\end{rem}

\begin{prop}\label{prop:Ztransitive} 

If $Z_1,Z_2$ are distinct irreducible components of $V(F)$ such that $\pi:Z_i\to \CC^k$ is dominant and the fiber is finite,
then the monodromy group of $\pi:Z_i\to \CC^k$ is transitive for each $i$, 
and the monodromy group of $\pi:Z_1\cup Z_2 \to \CC^k$ is not transitive. 

\end{prop}
\begin{proof}
See Proposition 2.5 of \cite{galois}.
\end{proof}

The main consequence of this proposition is that we can use homotopy continuation to populate the fiber if given a starting point.
This is described as a special case of Algorithm \ref{algo:standardGeneral} and has been exploited in numerous instances as mentioned in the introduction. 
We make the assumption that one solution to a system is easy to find; this is typically done by fixing some of the variables and solving for the parameters.
Using the terminology of \cite{Pirola}, if the monodromy group of
$\pi$ is the full symmetric group
$\mathfrak{S}_d$, then we say that
$\cG_{\pi}$ is \emph{uniform}. 

Our methods focus on monodromy groups that are imprimitive; in particular these are not uniform.

\begin{defn}\label{def:imprimitive}
Let $\cG$ be a group acting transitively on a finite set $[k]=\left\lbrace 1,\ldots,k\right\rbrace$. A subset $B \subseteq [k]$ is a \demph{block} if $gB=B$ or $gB \cap B = \emptyset$ for every $g \in \cG$. We say that $\cG$ is \demph{primitive} if its only blocks are $\emptyset, [k]$ and $\left\lbrace j \right\rbrace$ for $j \in [k]$. Otherwise, $\cG$ is \demph{imprimitive}.
\end{defn}

\subsection{Decomposing a projection}\label{ss:decomp}
We have the following definition from \cite{Pirola}. 

\begin{defn} \label{def:decomp}
Let
$\pi:Z\to Y$ be a generically finite dominant map of degree
$d$
between complex algebraic varieties. 
We say that
$\pi$
is (nontrivially)
\emph{decomposable}
if there exists an open
dense subset
$U\subseteq Y$
over which
$\pi$
factors as
\begin{equation}\label{eq:decompPi}
\pi^{-1}(U)\stackrel{\alpha}{\rightarrow}V\stackrel{\beta}{\rightarrow}U
\end{equation}
where
$\alpha$
and
$\beta$
are finite morphisms of degree at least two. 
If either $\deg\alpha=1$ or $\deg\beta=1$, then we say the decomposition $\pi=\beta \circ \alpha$  is trivial.
\end{defn}

We will be interested in the case where $Z$ is an irreducible component and curve in  $V(F)$ with $F:\CC\times \CC^n\to \CC^n$.
It follows that if $\pi$ is decomposable then there is an intermediate cover. 
This leads directly to the following proposition.

\begin{prop}\label{prop:structureDecomp}
The  projection $\pi: Z\to\CC$ is decomposable as in \eqref{eq:decompPi} if and only if the Galois group $G_\pi$ is imprimitive.
Moreover,  if $\pi:Z\to\CC$ is decomposable, then $\cG_\pi$ is a subgroup of a wreath product $\mathfrak{S}_a\propto\mathfrak{S}_b$ where $a=\deg\alpha$, $b=\deg\beta$,
and the Galois group $\cG_{\beta:V\to U}$ is a transitive subgroup of $\mathfrak{S}_b$.
\end{prop}
\begin{proof}
The first part is immediate with Galois theory by using the  one-to-one correspondence between the intermediate subfields of the field extension induced by $\pi$ and the subgroups of $\cG_\pi$. 

 The second part follows by our assumption that $Z$ is an irreducible curve and that the projection to $\CC$ is dense. 
\end{proof}

We will use Proposition \ref{prop:structureDecomp} in Algorithm~\ref{algo:standardGeneral}. 

\begin{defn}
    For a parametric polynomial system 
    $F:\upstairs\to \CC^n$ 
    define the projection map $\pi:V(F)\to \CC^k,\,(u,z)\mapsto u$.  
    We say $F$ is \demph{decomposable} w.r.t 
    $g:\CC^n\to \CC$
    if there exists a dense Zariski open subset $U$ of $\CC^k$
    such that $\pi$ factors as 
    \[
    \pi^{-1}(U)\stackrel{\alpha}{\rightarrow}  
    \alpha(\pi^{-1}(U))
    \stackrel{\beta}{\rightarrow}U
    \]
    where $\alpha ( u, z ) = (u, g(z))$, $\beta(u,y)=u$.
\end{defn}

In the definition, it is important that the first coordinates are consistently $u$, otherwise the composition $\beta \circ \alpha$ does not necessarily decompose $\pi$.
For most choices of $\alpha$, we have $\deg\alpha$ is one and will not yield a nontrivial decomposition of the projection. 
To find a nontrivial $\alpha$ (when they exist), one can employ algorithms in invariant theory or decomposition of polynomials.

\begin{rem}[Fundamental Invariants]\label{rem:invariants}
Given a parametric polynomial system $F$, one can ask how to search for a suitable polynomial function $g$ that makes $F$ decomposable.
When $F$ is invariant under a finite group action $\Gr$, 
this can be answered using fundamental invariants. We can take the polynomial map $g$ to be a random linear combination of a finite set of fundamental invariants. In practice, it is often enough to take $g$ to be a single low degree invariant, as illustrated in Subsection~\ref{subsec:gaussian}. 

To obtain fundamental invariants, one can use for instance the \emph{Reynolds operator} from computational invariant theory \cite{DK2002CIT,Sturmfels2008AIT}. The Reynolds operator takes a polynomial and maps it to an invariant: 
$f\mapsto \frac{1}{|\Gr|}\sum_{\sigma \in \Gr}  \sigma (f)$.
When the system $F$ is invariant under $\Gr$, which is usually easy to check, we can often find a $g$ to decompose $F$ by applying the Reynolds operator to a generic polynomial of suitable degree. For instance, in Example \ref{cyclicn} we see that the system is invariant under the dihedral group $D_n$, so we can use this technique. 
\end{rem}

\subsection{Homotopy continuation}\label{sec:ws}

Homotopy continuation is one of the central tools in numerical algebraic geometry. 
A homotopy uses a numerical predictor corrector method to deform a solution to one set of equations to another. 
We want to prescribe homotopies that take advantage of the structure of the system to improve computational performance.
This can be done in a number of ways.
For example, polyhedral methods use the Newton polytope structure of the system and regeneration uses the equation by equation structure. These methods have led to off the shelf software \cite{PHC,HomPS,homotopyjl}, and \cite{Bertini} respectively. 

A homotopy with path-parameter $t$, is given by
$H_T:\CC\times \CC^n\to\CC^n, (t,z)\mapsto H_t(z)$.

Witness sets are a fundamental data structure in numerical algebraic geometry to describe varieties.
The standard witness set consists of a witness point set, a linear space, and equations \cite{SWbook}. 
When the variety has more structure, additional information can be included in the witness set. 
This information can include multiplicity like when using deflation \cite{Deflation,LeykinDeflate} or multiprojective structure~\cite{HRmulti}. 

In many instances, one does not have  defining  equations for the ideal of a variety. 
One such case is when a variety is given by the image of a projection; these are described by pseudo witness sets \cite{HSpseudo} or witness sets of projections  \cite{Witness}. 
For irreducible curves, which is the case we reduced to in Section \ref{s:prelim}, 
we recall the concept of witness sets of projections.
Then, we introduce witness set factors for decomposable projections. These are much in the same vein as pseudo witness sets. They describe fibrations and sections of fiber bundles. 

\begin{defn}\label{def:witness}
Let $F:\upstairs\to \CC^n$ be a parametrized polynomial system, $Z \subseteq \CC^k\times \CC^n$ an irreducible subvariety and $\pi:\upstairs\to\CC^k$ be the projection given by $\pi(u,z)=u$.
The \demph{witness set} $W_{\pi}(Z)$ of $\pi$ restricted to $Z$,
 consists of the following three pieces of information:
$$\{
F, \,
q^{\star},\,
\Wpiq
\}\quad \text{ where } \Wpiq:=\{ z\in \CC^n : F(q^{\star},z  )= 0 \,\text{and}\, (q^{\star},z)\in Z  \}$$
where $q^{\star}$ is a general point in the image $\pi(Z)\subset \CC^k$. 
When the context is clear, we denote it simply by $\Wpi$.
The set $\Wpiq$ is said to be a \emph{witness point set} and its elements \emph{witness points}.
\end{defn}

With $\Wpi$, we are able to easily describe the fiber over another point $q' \in \CC^k$. 
From the witness set $\Wpi$, we use  a homotopy to deform $q^{\star}$ to $q'$ which  deforms the witness point set $\pi^{-1}(q^{\star})$.
Doing so, every nonsingular isolated point of the fiber $\pi^{-1}(q')$ will be a limit point of one of the deformed witness points \cite{SWbook}. 

When a projection $\pi = \beta \circ \alpha$ is decomposable, 
the witness set $\Wpi$ has extra structure that we capture with two witness point subsets:  
\begin{itemize}
    \item $\Wpia{q^\star} \subset \Wpiq$ consists of $\deg \alpha$ points that map to a single point under $\alpha$.
    \item $\Wpib{q^\star} \subset \Wpiq$ consists of $\deg \beta$ points that map to distinct points under $\alpha$.
\end{itemize}
We call such witness points subsets an \emph{$\alpha$-factor} and a \emph{$\beta$-factor} of $\Wpiq$, respectively.

\begin{rem}\label{cor:decWS}
Using the notation in Definition~\ref{def:witness},
suppose $\Wpi$ has a witness point set with $ab$ witness points and imprimitive Galois group ($a,b>1$).
Consider subsets  $A$, $B$ of $\Wpi$ consisting of 
 $a$ and $b$ distinct witness points of $\Wpi$ respectively. 
Then, $A$, $B$ are an $\alpha$-factor and a $\beta$-factor respectively for $\Wpi$ if and only if the following occur: \\
(1) $A$ is a block, i.e. for each $\gamma\in\cG_\pi$,  the intersection $\gamma\cdot A\cap A$ is empty or  $A$.\\
(2) $B$ is a set of representatives for the partition by the blocks $\left\lbrace \gamma A \right\rbrace _{\gamma\in\cG_\pi}$, i.e. for each $\gamma\in\cG_\pi$, the intersection $\gamma\cdot A\cap B$ is precisely one point. 
\end{rem}

\begin{example}
Recall the curve $Z$ from Example \ref{ex:first}. The witness set of $\pi:Z\to\CC$ consists of $2000$ points. 
For $u=-3$, 
the following are factors of the witness set $\Wpi(-3)$:
$$
\Wpia{-3} = \left\{i\zeta,i\zeta^2,\dots,i\zeta^{1000}\right\}
\text{ and }
\Wpib{-3} = \left\{i\zeta,\sqrt[1000]{-3}\right\},$$
where $\zeta$ is a primitive $1000$th root of unity. 
These sets consist of a thousand and two witness points, respectively.
Remark \ref{cor:decWS} leads to computational improvements  in the following sense. 
Suppose we were given  $\Wpia{-3}$ and wish to compute a $\beta$-factor $\Wpib{-3}$. 
In this example,  this means finding a point in $\Wpia{-3}$ and $\Wpib{-3} \setminus \Wpia{-3}$. 
Let $\gamma$ denote a path in $\CC \setminus\{0,1\}$.
It would be a waste of resources to use homotopy continuation to track every point of $\Wpia{-3}$ along $\gamma$. According to Remark \ref{cor:decWS}, if tracking $\Wpia{-3}$ along $\gamma$ produces an end point in $\Wpi \setminus \Wpia{-3}$, then  tracking any single representative of $\Wpia{-3}$ along $\gamma$ produces an end point in  $\Wpi \setminus \Wpia{-3}$.
\end{example}

The notion of witness set factors generalizes naturally to projections with more than two factors in their decomposition. 
Suppose $\pi$ decomposes as $\alpha_\ell \circ \cdots\circ\alpha_2 \circ \alpha_1$ with witness set $\Wpi$. An \emph{$\alpha_i$th-factor} of $\Wpi$ is a set $\Wpi(\alpha_i)$ of $\deg\alpha_i$ distinct witness points of $\Wpi$ satisfying the following properties: 
\begin{enumerate}
\item the map  $\alpha_{i-1} \circ\cdots\circ\alpha_{1}$ on  $\Wpi(\alpha_i)$ is injective,
\item the image of $\Wpi(\alpha_i)$ under $\alpha_i \circ \alpha_{i-1} \cdots\circ \alpha_1$ is one point. 
\end{enumerate}

\begin{example}
We illustrate the decomposition of a projection into multiple factors.
Let $\alpha_i(u,z)=(u,z^2)$ for $i=1,2,3$, $\alpha_j(u,z)=(u,z^5)$ for $j=4,5,6$, and 
$\alpha_l(u,z)=u$ for $\ell=7$.
Then, the projection $\pi:Z\to\CC$ from Example \ref{ex:first} decomposes into 
$\pi=\alpha_7\circ\alpha_6\circ\alpha_5\circ\alpha_4\circ\alpha_3\circ\alpha_2\circ\alpha_1$.
The witness point sets for each of these factors corresponding to $\alpha_1,\dots,\alpha_7$ consist of $2,2,2,2,5,5,5$ points respectively. 
\end{example}

In this article, we use monodromy to compute the witness set factors.

\section{Computing a fiber of a decomposable projection}
In this section we will give a general monodromy algorithm to populate a fiber and a trace test algorithm as a stopping criterion. 

\subsection{Decomposable monodromy algorithm}\label{ss:decmon}

Our aim is to compute a subset $S$ of the solution set to $F(u^{\star},x)=0$
for generic $u^{\star}$. We also want to make use of the information that the system is decomposable.

To this end, we present Algorithm~\ref{algo:standardGeneral}, which computes solutions via monodromy loops, but only keeps track of solutions that map to different images under a polynomial $g: \CC^n \rightarrow \CC$. 

We show that this algorithm is consistent with the setting of a parametrized polynomial system $F$ with respect to $g$ in Theorem~\ref{thm:decomposablemonodromy}.

\begin{algorithm}
\DontPrintSemicolon 
\KwIn{
    \;
    Parametric polynomial system:  $F:\upstairs\to \CC^n$.\;
    General parameters: $u^\star\in \CC^k$.\;
    Start solutions: A nonempty finite subset $S^\star$ of 
    $\{ x\in \CC^n : F(u^\star,z ) = 0  \}\subset\CC^n$.\;
    Polynomial map: $g:\CC^n\to \CC$.\;
    Stopping criteria:  $\bC$. 
}
\KwOut{A finite subset of $\{ z\in \CC^n : F(u^\star,z) = 0  \}\subset\CC^n$.}
\While{the criterion $\bC$ is \texttt{False} \;}{
    Let $\branchLocus\subset\CC^k$ be the branch locus. \;
    Set $\gamma$ to be a loop in $\CC^k\setminus \mathcal{B}$ beginning at $u^\star$. \;
	Do a parameter homotopy along $\gamma$  with start points $S^{\star}$ to obtain endpoints $E$. \;
	\For{ each point $p$  in $E$  \; }{
		\If{$g(p)\not\in g(S^{\star}))$\label{step:lose} \;}{$S^{\star}\leftarrow S^{\star}\cup \{ p \}$.\; }
		}	
	}
\Return{ $S^{\star}$.}\;
\caption{Decomposable monodromy algorithm}
\label{algo:standardGeneral}
\end{algorithm}

\begin{rem}[Trivially decomposable]
    When $g$ is the identity map, i.e., $g(z)=z$ 
    this is the standard monodromy as seen in \cite{polsyshommon,Monodromy}. 
    We call this the \demph{Classical Monodromy Method}. 
\end{rem}

\begin{theorem}\label{thm:decomposablemonodromy}
Using the notation in Algorithm~\ref{algo:standardGeneral},
if $u^{\star}$ is general and $S^{\star}$ is a single point, 
then the output is contained in a unique irreducible component $Z$ of $V(F)$. 
Moreover, there exists a sequence of loops such that the output is $W_\pi(u^{\star}).$
\end{theorem}

\begin{proof}
We have that the monodromy group of $F$ acts transitively on $W_\pi(u^{\star})$ for a general point $u^{\star} \in \pi(Z)$ by Proposition~\ref{prop:Ztransitive}. Hence, if $g$ is the identity map, there exists a sequence of loops where the output is the entire set of solutions
$$\{ z\in \CC^n : F(u^\star,z) = 0 \text{ and } (u^{\star},z)\in Z\}.$$ For an arbitrary $g$, the output need not return this entire set of solutions: if the distinct solutions $p,q$ are such that $g(p)=g(q)$, then Algorithm~\ref{algo:standardGeneral} only returns one of them according to step~\ref{step:lose}. Indeed, the algorithm returns a set of solutions that have distinct images under $g$. This is precisely what is needed to have a witness set
$$\Wpi(u^\star)=\{ z\in \CC^n : F(u^{\star},z  )= 0 \,\text{and}\, (u^{\star},z)\in Z  \},$$
by definition of decomposability with respect to $g$.
\end{proof}

\begin{example}
Consider the parametric polynomial system given by the equation
$F = z^6+z^4+z^2+u =0$.
This is decomposable with respect to $g(z)=z^2$. The image under $\alpha(u,z)=(u,z^2)$ is given by the curve in $\CC^2$ defined by 
$$y^3+y^2+y+u=0.$$
The steps of Algorithm~\ref{algo:standardGeneral} are illustrated in Figure~\ref{fig:triple-coordinates}.
The loop $\gamma$ is given by the unit circle in $u$-space beginning at the red point of Figure~\ref{fig:triple-coordinates}(c). 
This loop avoids the branch locus 
$V(u(27u^2-14u+3))$.

Tracing around $\gamma$ once lifts to a path in the $z$-complex plane and $y$-complex plane connecting two red points. 
So tracing around $\gamma$ six times gives one revolution around the ``square" in Figure \ref{fig:triple-coordinates}(a) but two revolutions around the ``triangle" in Figure \ref{fig:triple-coordinates}(b). If $g(z)=z$ is the identity, the algorithm 
will return all six solutions. 
If $g(z)=z^2$, the algorithm returns three solutions which have distinct $y$-coordinates. 
\end{example}

\begin{figure}[h]
\centering
\subfigure[]{
\includegraphics[scale=0.6]{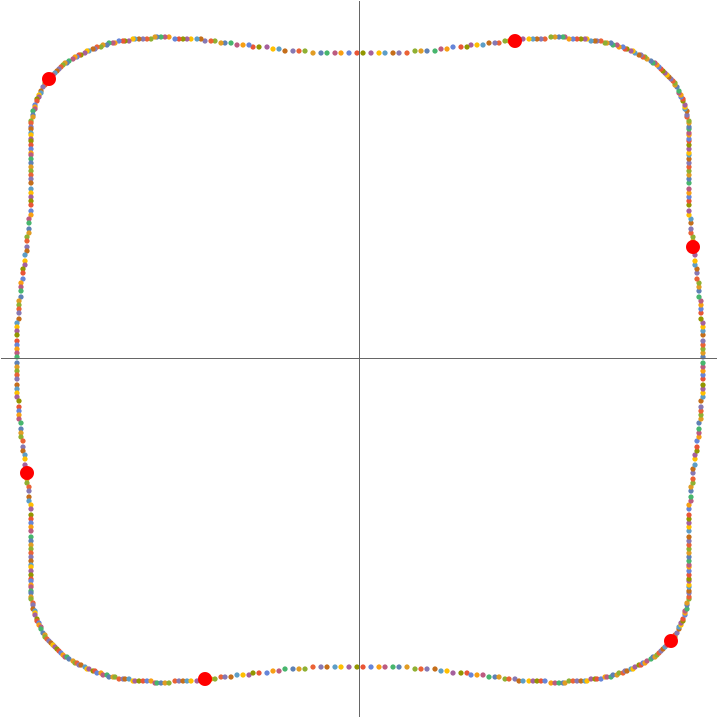}
\put(-220,210){{$z$-complex plane}}
}
\subfigure[]{
\includegraphics[scale=0.6]{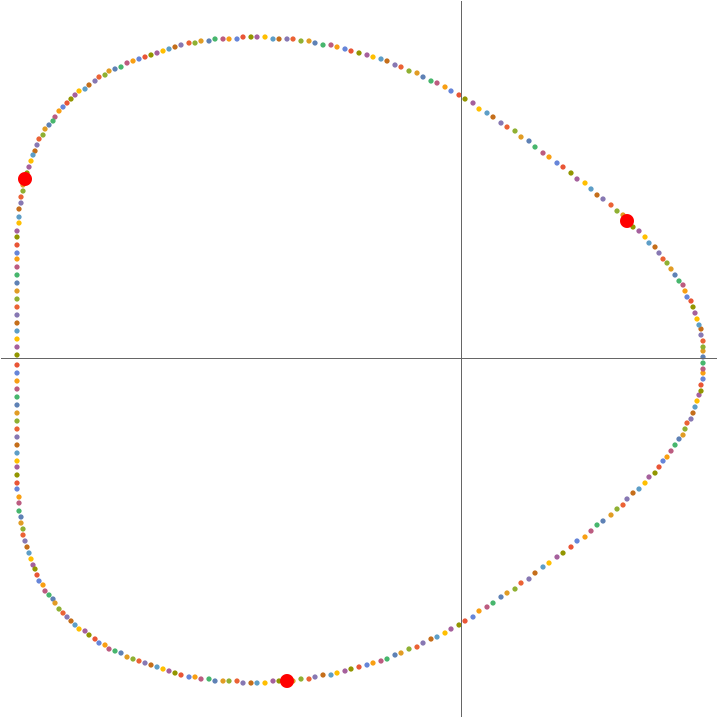}
\put(-220,210){{$y$-complex plane}}
}
\subfigure[]{
\includegraphics[scale=0.6]{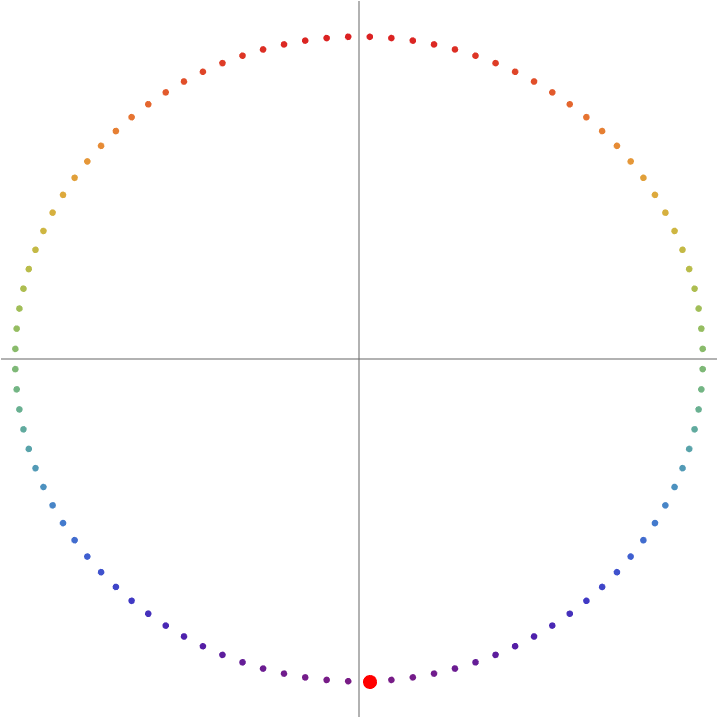}
\put(-111,105){{$\star$ origin }}
\put(-90,84){{$\star$  
{\tiny $0.259-0.210\sqrt{-1}$
}}}
\put(-90,136){{$\star$ {\tiny $0.259+0.210\sqrt{-1}$}}}
\put(-220,210){{$u$-complex plane}}
}
\caption{Illustration of Algorithm~\ref{algo:standardGeneral} 
}
\label{fig:triple-coordinates}\end{figure}

\begin{rem}[Choosing loops]
While taking a loop given by concatenating random line segments avoids the branch locus with probability one, a systematic choice of loops can lead to computational savings during implementation.
\end{rem}

\begin{rem}[Stopping criteria implementation]
The framework that we have provided is flexible and allows for numerous types of stopping criteria. 
Some typical criteria count the number of times the \texttt{while} loop is entered, an upper bound on the number of solutions, or the wall time of the computation. 
In our implementation, the default stopping criteria are:
(1) the while loop did not find any new solutions ten times in a row
 or 
(2) the number of solutions found equals or exceeds $D$ which is chosen accordingly.
\end{rem}

\subsection{Elementary examples}\label{ss:examples}
\begin{example}\label{ex:firstFollowUp}
Following up on   Example \ref{ex:first},  
consider the curve $\cC\subset\CC^2$ defined by $z^{2000}-2z^{1000}+u$. 
The branch locus for the projection  $\pi:\cC\to\CC$ consists of two points: $u=0$ and $u=1$.
The critical locus can be written as the  intersection of two ideals:
$$
(z^{2000}-2z^{1000}+u,2000z^{1999}-2000z^{999})=(u,z^{999})\cap(u-1,z^{1000}+1).
$$
Let $\gamma_0,\gamma_1$ denote loops based at a general point of $\CC$ that encircle $u=0$ and $u=1$ respectively.
These loops induce the following two permutations:
$$
(1,2)(3,4)(5,6)\cdots(1999,2000)\quad{\text{and}}\quad(1,3,5,\dots,1999)(2,4,6,\dots,2000).
$$
From this set of generators,
one sees the Galois group $\cG_\pi$ is imprimitive: the odd and even numbered solutions form two nontrivial blocks. Indeed, note that the odd (even) numbered solutions are permuted amongst themselves or are taken to even (odd) numbered solutions. 
\end{example}

\begin{example}\label{ex:mv}
Consider the parameterized system of equations $F:\CC^6\times\CC^2\to \CC^2$
given~by
\begin{equation}\label{eq:mv}
(u_1+u_2(z_1z_2^2+z_1^2z_2))z_1z_2=0\quad\quad
u_4+u_5(z_1+z_2)+u_6z_1z_2=0.
\end{equation}

We are only interested in the component $Z$ of $V(F)$ not contained in a coordinate hyperplane. 
The projection $\pi$ of $Z$ to $\CC$ has a fiber of four solutions corresponding to the mixed volume of the system.
Let $\alpha(u,z_1,z_2)=(u,z_1+z_2)$.
Using decomposable monodromy, we find $\Wpib{u}$ has two points, which is less than the mixed volume.
\end{example}

\subsection{A trace test stopping criterion}\label{ss:tracetest}
An immediate question in regards to Algorithm~\ref{algo:standardGeneral} 
is what can we use as a stopping criterion. 
To answer this, we have our next algorithm, which involves a trace test~\cite{tracetest,HRmulti,TraceTesto}. 

However, since we cannot use a trace test directly on the parametric system $F: \upstairs \to \CC^n$, we will modify it to get a new system. 

Let $\bar F $ be the system of $n$ polynomial equations in $n+1$ unknowns $(v,z) \in \CC \times \CC^{n}$ given~by 
\begin{equation}\label{eq:Fbar}
\bar F(v, z) := F(L(v), z)
\text{ where }  
L:\CC\to \CC^k \text{ is a general affine linear function}.
\end{equation}
Clearly, $\bar F$ is underdetermined.

\newcommand{\constraintA}{A}

To get a parametric square system of equations, 
we introduce a new constraint and a single parameter $t$.
For generic linear functions $\ell_1,\ell_2:\CC\to \CC$ we construct the following nonlinear constraint that depends on 
$t$ and a polynomial $g(z)$:
\begin{equation}\label{eq:A}
\constraintA(t,v,z) =  \ell_1(v)\ell_2(g(z))-t.
\end{equation}

 The one parameter system of equations 
\begin{equation}\label{eq:Fbar}
\bar F(v,z)=0 , \constraintA(t,v,z)=0
\end{equation}
is a square system. 
When $t=0$, 
the system factors as 
\begin{equation}\label{eq:Fell}
\bar F(v,z)=0,\, \ell_1(v)=0
\text{  and }
\bar F(v,z)=0,\, \ell_2(g(z))=0.
\end{equation}

The former system corresponds to the solutions of  $F(u,z)=0$ where $u=L(v)$ and $v$ is fixed.  
The latter corresponds to solutions of $F(u,z)=0$ with the parameters restricted to the line $u=L(v)$ and $g(z)$ fixed to be a generic value.

Algorithm~\ref{algo:standardGeneral}
(or some other black box polynomial system solving method)
can be used to find subsets of the set of solutions for each of the systems by introducing the parameters  $t_1,t_2$:
\begin{equation}\label{eq:monodromyFell}
\bar F(v,z)=0,\, \ell_1(v)=t_1
\text{  and }
\bar F(v,z)=0,\, \ell_2(g(z))=t_2.
\end{equation}
Since $\ell_i$ is a generic affine linear function, the parameters $t_1=0$ and $t_2=0$ are generic.

Denote by  $S$ a subset of solutions $W$ to 
$\bar F(v,z)=0 , \constraintA(t,v,z)=0$
with $t=0$. 
For $t=\pm 1$
Let $S(\pm1 )$ denote the set of endpoints of a homotopy with $t$ varying from 0 to $\pm 1$ with start points $S$.

With the sets of points $S(0):=S$, $S(-1)$, and $S(1)$
we can do a trace test to verify we have found all of the solutions. 
 
\begin{algorithm}
\DontPrintSemicolon
\KwIn{
    \;
    Nonempty subsets $S(0)$, $S(-1)$, and $S(1)$ in $\CC^{n+1}\cong\CC\times \CC^n$\;
    A polynomial map $g:\CC^n\to \CC$
}
\KwOut{
    A nonnegative real number\;
    }
Denote by $\Tr(j)\in \CC^3$ the coordinate-wise average of the set of points 
$\{ (v,g(z),vg(z)) : (v,z)\in  S(j)
\}$.\;
Set $\epsilon$ to  $
||(\Tr(1)-\Tr(0)) - (\Tr(0)-\Tr(-1))  ||.
$
\;

\Return{ $\epsilon$}\;
\caption{Pseudo-Segre Trace test}
\label{algo:segreTrace}
\end{algorithm}

We call this a \emph{Pseudo-Segre} trace test because we use the image coordinate $g(z)$ as opposed to just $z$, and we use the traces of a curve in a Segre embedding. 
\begin{rem}[Degree of an affine curve]\label{rem:affineTT}
    A special case of the previous algorithm is a classic technique in numerical algebraic geometry to verify the degree of an irreducible affine curve. Let $X$ be a curve in $\mathbb{C}^n$ that is an irreducible component of 
    $V(h_1,\dots,h_{m})$ and let $z^\star$ denote a generic point of $X$.
    For generic $t^\star\in \CC$, and generic $(b_1,\dots, b_n)\in \CC^n$ apply Algorithm \ref{algo:standardGeneral} to the parameterized polynomial system 
    \[
    F(t,z) =   [ h_1(z)=0  ,  \dots, h_{m}(z)=0,  
    \sum_{i=1}^n b_i (z_i-(z^\star)_i)   =t-t^{\star} 
    ]
    \]
    with $g(z)=z$.
    It has been shown in \cite{TraceTesto,tracetest} that Algorithm~\ref{algo:segreTrace} returns $\epsilon=0$ if and only if $S(0)=W$.

    We recall that in our general situation for decomposable parametric polynomial systems, the standard trace test cannot be applied directly.
\end{rem}

\begin{theorem}
We use the preceding notation in this subsection and assume $V(F)$
is irreducible. 
Let $W\subset \CC^{n+1}$ be the set of solutions to the system 
\[
\bar F(v,z)=0, \quad \constraintA(0,v,z)=0.
\]
Let $\psi:\CC^{n+1}\to \CC^3$ be the map 
$\psi(v,z) = (v,g(z), vg(z))$, 
and suppose $\psi$ restricted to  
$S(0) \subseteq W$ 
is one to one. 
Then $\psi(S(0)) =   \psi(W)$ 
if and only if
the output of Algorithm~\ref{algo:segreTrace} is 
$\epsilon = 0$.
\end{theorem}

\begin{proof}
The main idea of the proof is to reduce to the the case where we are using a trace test to verify the degree of an affine curve in $\CC^3$ like in Remark~\ref{rem:affineTT}.

The polynomial system $F(u,z)=0$ defines an irreducible variety in $\CC^k\times \CC^n$.
In \eqref{eq:Fbar} we restrict the parameter space $\CC^k$ to a general line using $L:\CC\to\CC^k$  parameterized by $v$. By treating $v$ as an unknown, 
we have $\Fbar(v,z)=0$ defining an algebraic variety in $\CC^{n+1}\cong\CC\times \CC^n$, which is in fact an irreducible curve by Bertini's theorem.

Recall $\ell_1(v)$ and $\ell_2(g(z))$ from~\eqref{eq:A}.
Define $Y$ to be the graph of the map 
\[
\pi: V(\Fbar)\to \CC,
\quad 
(v,z)\mapsto t= \ell_1(v) \ell_2(g(z)).
\]
The variety $Y$ is irreducible because $V(\Fbar)$ is irreducible. Moreover, $V(\Fbar)$ is defined by the system ~\eqref{eq:Fbar}.
Thus, $W$ is precisely  $\pi^{-1}(0)$.

On the other hand, using $g$, we map $V(\Fbar)$
into $\CC^2\cong\CC\times \CC$ by 
$\bar \alpha (v,z):=(v,g(z))$.
The variety $\bar \alpha(V(\Fbar))\subset\CC^2$ has  coordinate projections $\pi_1(v,g)=g,\pi_2(v,g)=v$.
The degrees of these coordinate projections are say $d_1$ and $d_2$ respectively. 
So we can assume the curve $\bar \alpha(V(\Fbar))\subset\CC^2$
is defined by a bivariate  polynomial  in $(v,g)$ with bidegree $(d_1,d_2)$.

Since $\ell_1,\ell_2:\CC\to \CC$ 
are general affine linear functions, it follows
\[
\bar \alpha(V(\Fbar))\cap 
V(\ell_1(v)\ell_2(g))
\subset\CC^2
\]
consists of $d_1+d_2$ points. 
More importantly, 
by taking an affine chart of a Segre embedding,  we map $\CC^2\to \CC^3$ by $\sigma(v,g)=(v,g,vg)$. 
With this embedding the bi-degree $(d_1,d_2)$ curve $\bar \alpha(V(\Fbar))$ is a degree $d_1+d_2$ curve in $\CC^3$.  
Moreover, in the $\CC^3$ coordinates and for $t\in \CC$, 
the bilinear constraint $\ell_1(v)\ell_2(g)=t$ defines a general hyperplane $H_t$ in $\CC^3$.

In summary the intersection points of the curve $\sigma(V(f))$ with hyperplane $H_t$ in $\CC^3$
is $\psi(W)=\sigma(\bar \alpha (W))$, 
which is the intersection points of an affine curve with a hyperplane. 
We can use a standard trace test from Remark~\ref{rem:affineTT} to verify we
have found all points of intersection in $\CC^3$, i.e., that $\psi(S(0)) =   \psi(W)$. The (exact) trace test is successful if and only if $\epsilon = 0$ as mentioned in Remark~\ref{rem:affineTT}.

This completes the proof because we assume $\psi$ restricted to $S(0)$ is one to one. 
\end{proof}

\begin{rem}[Separable Solve Method]\label{rk:separable}
    Since the polynomial system $\bar F$
    in equation~\ref{eq:monodromyFell} factors when $t=0$, we exploit this fact to solve 
    these two systems independently using  
Algorithm~\ref{algo:standardGeneral}, followed by Algorithm~\ref{algo:segreTrace}. 
Empirically, this speeds up performance by a factor of two for difficult problems and we call this the \demph{Separable Method} in our computational~results.
\end{rem}

\section{Applications}\label{sec:exShort}

In the first subsection we have a case study on a moment system. 
In the second subsection we present computational results motivated by kinematics. 
In the last subsections, we have a case study on the cyclic $n$-roots problem up to $n=9$.

\subsection{Case Study: Gaussian Mixtures}\label{subsec:gaussian}

An example from statistics where polynomial systems with symmetry arise naturally is the moment equations of Gaussian mixture distributions. For history and context of this problem, see \cite{AFS}. 

The first non-trivial instance of this problem involves the five moment equations corresponding to a mixture of two univariate Gaussians:

\begin{equation}
\label{eq:karlpara}
 \begin{matrix}
m_0 & = & \lambda_1 + \lambda_2 \\ 
m_1 & = & \lambda_1 \mu_1 + \lambda_2 \mu_2 \\
m_2 & = & \lambda_1 (\mu_1^2 + \sigma_1^2) + \lambda_2 (\mu_2^2 + \sigma_2^2) \\
m_3 & = & \lambda_1 (\mu_1^3 + 3 \mu_1 \sigma_1^2) + \lambda_2 (\mu_2^3 + 3 \mu_2 \sigma_2^2) \\
m_4 & = & \lambda_1 (\mu_1^4 + 6 \mu_1^2 \sigma_1^2 + 3 \sigma_1^4) 
          + \lambda_2 (\mu_2^4 + 6 \mu_2^2 \sigma_2^2  + 3 \sigma_2^4) \\
m_5 & = & \lambda_1 (\mu_1^5 + 10 \mu_1^3 \sigma_1^2 + 15 \mu_1 \sigma_1^4) 
          + \lambda_2 (\mu_2^5 + 10 \mu_2^3 \sigma_2^2  + 15 \mu_2 \sigma_2^4). \\
\end{matrix}
\end{equation}
 
The indeterminates are $\lambda_1,\lambda_2, \mu_1, \mu_2, \sigma_1^2, \sigma_2^2$, and $m_0, m_1,m_2,m_3,m_4,m_5$ are the parameters, which correspond to given numerical moments. Note that if we have a solution  $(\lambda_1,\lambda_2, \mu_1, \mu_2, \sigma^2_1, \sigma^2_2)$, then  $(\lambda_2,\lambda_1, \mu_2, \mu_1, \sigma^2_2, \sigma^2_1)$ is also a solution. This phenomenon is known in statistics as $\lq$label-swapping'. We claim that this symmetry corresponds to a map decomposition of the projection of the incidence variety defined by the system to the moment space. 
In general, for a $k$ mixture model, the $\ell$th moment equation is given by
\begin{align}
    m_{\ell} &= \lambda_1 M_{\ell}(\mu_1, \sigma_1)+\cdots+ 
    \lambda_k M_{\ell}(\mu_k, \sigma_k)
    \label{momenteq}
\end{align}

where $M_\ell (\mu_i, \sigma_i)$ can be calculated recursively as $M_0(\mu_i, \sigma_i) = 1$, $M_1(\mu_i, \sigma_i) = \mu_i$ and $M_\ell (\mu_i, \sigma_i) = \mu_i M_{\ell - 1} + (\ell-1) \sigma_i^2 M_{\ell - 2}$ for $\ell \geq 2$. The \emph{$k$ mixture moment problem} is to find all isolated solutions defined by the system of polynomials in  equations $m_\ell$ for $\ell = 0,\ldots, 3k-1$. Due to the label-swapping symmetry discussed above, $\mathfrak{S}_k$ acts on the solution set given by a $k$ mixture system, partitioning the solutions into equivalence classes of size $k!$.

Beginning with the computation for a mixture of $k=2$ univariate Gaussians and restricting the parameters to a general line 
yields the curve $\cC$. Setting $y=\mu_1+\mu_2$  and eliminating the coordinates 
$\mu_1,\mu_2,\sigma_1^2,\sigma_2^2,\lambda_1,\lambda_2$ is nontrivial. 
Using a combination of substitutions and resultants, after three hours we found the defining equation for $\alpha(\cC)$. This polynomial is dense in bidegree $(9,9)$ consisting of $100$ terms. 

On the other hand, with standard monodromy we tracked 66 paths and 18 complex solutions are obtained. Moreover, if we use Algorithm \ref{algo:standardGeneral} instead with $\alpha(u, \lambda,\mu,\sigma)=(u, \mu_1+\mu_2)$ to decompose the map, we obtain 9 solution classes (of size 2) tracking only 24 paths in a particular instance. 

We also run the analogous computation for a mixture of $k=3$ univariate Gaussians. This includes the variables $\lambda_3,\mu_3,\sigma^2_3$ to the six equations in system (\ref{eq:karlpara}), and we need to include three more moment equations $m_6,m_7,m_8$ to make the system zero-dimensional.  This yields 225 equivalence classes of size $6=3!$ when using the general coordinate $\alpha(\boldsymbol{\lambda},\boldsymbol{\mu},\boldsymbol{\sigma})=\mu_1+\mu_2+\mu_3$. This number coincides with the one found via Gr\"obner bases in \cite{AFS}. 
For general $k$, one has $3k$ variables and a corresponding system of $3k$ moment equations. The fact that this yields a finite number of solutions for generic moments was proved in \cite{algident}. For $k=4$, the conjectured structure of the solutions to the system of twelve variables and twelve equations according to \cite{AFS} consists of 264600 complex solutions arranged in 11025 equivalence classes of size $4!=24$. Combining Algorithms \ref{algo:standardGeneral} and \ref{algo:segreTrace} we are able to disprove this conjecture. 
\begin{compresult}
\label{res:number_sols_gaussian_4_mix}
The number of 
solutions for a $k = 4$ 
mixture model is $248400 = 10350 \cdot 4!$
for generic moments $(m_0,m_1,\dots,m_{11})$.
\end{compresult}
\begin{proof}[Method]
Let $\Fbar:\CC\times\CC^{13}\to \CC^{13}$
be the parametric system in the unknowns $(v,\boldsymbol{\mu},\boldsymbol{\sigma^2})$ 
given by 

\[
\Fbar(t; v, \boldsymbol{\mu},\boldsymbol{\sigma^2})
=\begin{cases}
m_{\ell} - 
(\lambda_1 M_{\ell}(\mu_1, \sigma_1)+\cdots+ \lambda_4 M_{\ell}(\mu_4, \sigma_4)  & \ell=0,\dots11\\
\ell_{1}(v)\cdot \ell_{2}(g(\boldsymbol{\mu},\boldsymbol{\sigma}))+t
\end{cases}
\]
with $L(v)=(m_0,m_1,\dots,m_{11})$ where $L:\CC\to \CC^{12}$ is a general affine linear function and $\ell_1,\ell_2:\CC \to \CC$ are general affine linear functions. 

For $t=0$, 
we find $31815$ solutions and verify this is a complete set of solutions 
up to symmetry using Algorithm~\ref{algo:segreTrace}  with $\epsilon < 10^{-12}$.
Of the $31815$ solutions, $10350$ solutions satisfy $\ell_1(v)=0$.
Since $\ell_1$ is a general affine linear function, all of these solutions have the same $v$-coordinate, say $v^{\star}$.
The $10350$ are solutions for the moment system chosen as $L(v^{\star})$.
\end{proof}

We give computational results for Gaussian $2,3$ and $4$ mixture models in Table~\ref{table:gaussian_mixture_computations}. All computations were performed using \texttt{HomotopyContinuation.jl} \cite{homotopyjl} on a 2018 Macbook Pro with 2.3 GHz Quad-Core Intel Core i5 processor. The timings and number of monodromy loops are an average of $5$ trials. We initiate the trace test in Algorithm~\ref{algo:segreTrace} once there are $10$ loops with no new solutions.

\begin{table}[h!]
\begin{center}
\begin{tabular}{ |c|| c | c| c | c| } 
\hline
$k$& & $2$ & $3$ & $4$ \\ 
\hline
\hline 
 \multirow{2}{8em}{Total degree}& $\#$ paths  & $720$ & $ 362,880 $ & $ 479,001,600$ \\ 
 & time (s) & $0.38$ & $818.60$ &  Est. $3$ days
 \\ 
\hline
 \multirow{2}{8em}{Polyhedral} & $\#$ of paths  & 35 & $4271 $ & \multirow{2}{8em}{cannot compute start system}\\ 
& time (s) & $0.04$ & $22.37$ &  \\ 
\hline 
 \multirow{2}{8em}{Monodromy\\ (degree)} & $\#$ of paths  &  $17$ & $718$   & $32359.4$ \\ 
& time (s) & $0.02$ &  $1.45$ & $166.15$  \\ 
\hline 
\multirow{2}{8em}{Monodromy\\ (trace test)} & $\#$ of paths & $774$ & $15703$ & $921619.25$ \\
& time (s) & $3.27$& $55.66$ & $5592.13$ \\
\hline 
\multirow{2}{8em}{Monodromy\\(separable)} & $\#$ of paths & $700.8$ & $11212.2$  &  $627567.6$ \\
& time (s) & $7.75$   & $29.43$   & $3988.64$\\
\hline 
\end{tabular}
\end{center}
\caption{Number of paths tracked and average time to compute all solutions to Gaussian $k$ mixture moment equations using different homotopy continuation algorithms.}
\label{table:gaussian_mixture_computations}

\end{table}

In Table \ref{table:gaussian_mixture_computations}, the row \emph{Monodromy (degree)} corresponds to applying Algorithm~\ref{algo:standardGeneral} to the system $F$ with stopping criterion once the number of solutions reaches the degree. The row \emph{Monodromy (trace)} corresponds to applying Algorithm~\ref{algo:standardGeneral} to the system $\bar F$ and then doing a trace test with Algorithm~\ref{algo:segreTrace}. The row \emph{Monodromy (separable)} follows the algorithm as explained in Remark~\ref{rk:separable}.

Overall, we observe that Monodromy (degree) majorly outperforms the standard total degree and polyhedral homotopy continuation algorithms. We also see that polyhedral homotopy initially outperforms Monodromy (trace test), but once $k = 4$ polyhedral homotopy becomes computationally untenable. 

\subsection{Algebraic kinematics}
There will be two ideas illustrated in this example. First, decompositions of projections can have physical meaning in kinematics. Second, even with partial information, we are able to construct an $\alpha$ for decomposing the projection.

In this subsection we consider four-bar linkages and Alt's nine-point problem \cite{Alt,NinePoint}.
The first linkage is grounded in the plane at 
the endpoints $\ba_1:=(a_1,\bar a_1)$ and $\ba_2:=(a_2,\bar a_2)$; these endpoints are called the ground pivots. 
Two links with lengths $\ell_1$ and $\ell_2$ will be attached to the respective ground pivots $\ba_1$ and $\ba_2$; the position of the endpoints of these two links are denoted by $\bb_1:=(b_1,\bar b_1)$ and $\bb_2:=(b_2,\bar b_2)$.
The middle linkage is a (coupler) triangle $\overline{{\bb}_1{\bb}_2{\bp}}$ with $\bp:=(p,\bar p)$ called the coupler point of the four bar mechanism.
 The motion of ${\bf{p}}$ is coupled with the motion of the other two linkages.
The angle of motion of links $\overline{\ba_1\bb_1}$, $\overline{\ba_2\bb_2}$, and $\overline{{\bb}_1{\bb}_2{\bp}}$ are given by motion indeterminants  
  $(\theta_1,\bar\theta_1)$, $(\theta_2,\bar\theta_2)$ and  $(\phi,\bar\phi)$ respectively. 
  The motion indeterminants satisfy the angle relations  $\theta_i\bar\theta_i=1$ and  $\phi\bar\phi=1$ and vector loop relations

\begin{equation}\label{vectorLoop}
\begin{array}{ccccc}
\ell_{1}\theta_{1} & = & p+\phi b_{1}-a_{1},\quad & \ell_{1}\bar{\theta}_{1} & =\bar{p}+\bar{\phi} \bar{b}_{1}-\bar{a}_{1},\\
\ell_{2}\theta_{2} & = & p+\phi b_{2}-a_{2},\quad & \ell_{2}\bar{\theta}_{2} & =\bar{p}+\bar{\phi}\bar{b}_{2}-\bar{a}_{2}.
\end{array}
\end{equation}

Thus, the family of four bar linkages (with coupler point $\bp$ and motion) has  twelve
configuration indeterminants $\bK:=(\bp,\ba_1,\ba_2,\bb_1, \bb_2,\ell_1,\ell_2)$ and six motion indeterminants 
$\bM:=(\theta_1,\bar\theta_1,\theta_2,\bar\theta_2,\phi,\bar \phi)$
satisfying \eqref{vectorLoop} and  the  angle relations.

Projecting the family of four bar linkages to the configuration space yields a hypersurface defined by the polynomial 
$f_{cc}(\bp,\ba_1,\ba_2,\bb_1,\bb_2,\ell_1,\ell_2)$ found in Eq. (3.20) in \cite[Section~3.2]{kinAlg}. 
The degree of this polynomial with respect to $\bp$ is six.   
This means, when the indeterminants $\ba_1,\ba_2,\bb_1,\bb_2,\ell_1,\ell_2$ are fixed, the polynomial 
defines a degree six (coupler) curve in the $\bp$ plane. This curve is the set of points through which the coupler point passes through over the range of motions. 

If we restrict $p,\bar p$ to a line parameterized by $v$, then we have a monic univariate polynomial in $v$ whose coefficients are rational functions in the configuration indeterminants. We identify these coefficients with $Y_0,Y_1,Y_2,Y_3,Y_4,Y_5$ in the equation 
\eqref{couplerCurve}.
\begin{equation}\label{couplerCurve}
f_{cc}(L(v),\ba_1,\ba_2,\bb_1, \bb_2,\ell_1,\ell_2)=S^6+Y_5S^5+Y_4S^4+Y_3S^3+Y_2S^2+Y_1S+Y_0
\end{equation}
where
$L:\CC\to \CC^2 $  is a general affine linear function.
A general coupler curve is determined by the values of these six $Y$-coordinates. Since the polynomial is of degree six, there is a degree six map from the family of four bar linkages to the coupler curve space given by $Y$-coordinates. 
We denote this map by $\alpha'(\bK,\bM)$.

Alt's problem is to find the number of coupler curves that pass through a specified nine general points in the plane 
$\bd_i:=(d_i,\bar d_i)$ for $i=1,2,\dots,9$.
One formulation of the problem is to solve  the nine equations $f_{cc,i}(\bd_i,\ba_1,\ba_2,\bb_1, \bb_2,\ell_1,\ell_2)$,  where $\bp$  is set to random points 
 $\bd_i$ for $i=1,2,\dots,9$ in the plane, along with the vector loop relations and angle relations. 
The number of solutions has been found numerically to be $3!\times1442$. The $3!$ comes from the \emph{Robert's cognates} and label swapping symmetry. 
Thus, we can consider the $3!\times1442$ as the degree of the fiber of the projection $\pi$ of the incidence variety of four bar linkages going through nine points to the space of nine points; the incidence variety is in the configuration indeterminants $\bK$, motion indeterminants $\bM$, and indeterminants $\bd_i$ for $i=1,2,\dots,9$.

What we have discussed shows  the projection $\pi$ decomposes into $\alpha \circ \beta$ where 
\[
\alpha(\bd_1\dots,\bd_9; \bK,\bM):=(\bd_1,\dots\bd_9; \alpha'(\bK,\bM))
\text{ and } \beta(\bd_1,\dots\bd_9; Y_0,\dots,Y_5,)=(\bd_1,\dots\bd_9).
\]
Thus, we can use decomposable monodromy to determine a $\beta$ witness set. 
Indeed, simplifying the decomposition by restricting $\bd$-space to a line and taking $\alpha'(\bK,\bM)=Y_5$ we use Algorithm~\ref{algo:standardGeneral} and Algorithm~\ref{algo:segreTrace}  to recover the $1442$ different coupler curves.
In our computation, we  only tracked  $5028$ paths, which is even less than  $3!\times 1442$.

\subsection{Benchmarks with cyclic n-roots} \label{cyclicn}

\begin{table}[htb!]
\begin{center}
\begin{tabular}{ |c|| c | c| c | c|c|c| } 
\hline
$n$& & $5$ & $6$ & $7$ & $8$ & $9$  \\ 
\hline
\hline 
 \multirow{2}{8em}{Total degree}& $\#$ paths  & 120 & $ 720 $ & $5040$ & $40320$ & $362880$  \\ 
 & time (s) & $0.047$ & $0.448$ & $4.682$ & $53.269$ & $677.038$  \\ 
\hline
 \multirow{2}{8em}{Polyhedral} & $\#$ of paths  & $70$ & $156$ & $924$ & $2560$ & $11016$  \\ 
& time (s) & $0.042$ & $0.119$ & $0.717$ & $3.507$ & $21.917$    \\ 
\hline 
 \multirow{2}{8em}{Monodromy (degree)} & $\#$ of paths  & $21.2$ & $38.4$ & $437$ & $483$ & $4294$   \\ 
& time (s) & $0.008$ & $0.038$ & $0.223$ & $0.515$ & $5.567$ \\ 
\hline 
\multirow{2}{8em}{Monodromy (trace test)} & $\#$ of paths & $3664.2$ & $11725.2$ & $97226.2$ & $377716.8$ & $3031284$  \\
& time (s) & $3.450$ & $6.544$ & $81.124$ & $340.425$ & $4854.535$  \\
\hline 
\multirow{2}{8em}{Monodromy (separable)} & $\#$ of paths & $1005.8$ & $3656.8$ & $42034.4$ & $286364.6$ & $673893.6$ \\
& time (s) & $9.562$ & $10.171$ & $66.366$ & $216.27$ & $1118.85$\\
\hline 
\end{tabular}
\end{center}
\caption{Number of paths tracked and average time to compute all solutions to cyclic-$n$ roots using different homotopy continuation algorithms.}\label{table:cyclicroots}
\end{table}

One of the benchmark systems in polynomial system solving is the cyclic
$n$-roots problem. The system has variables $x_0,x_1,\ldots,x_{n-1}$ and parameters $u_0,u_1,\ldots,u_{n-1}$:
\begin{equation}\label{eq:cyclicn}
\begin{array}{lll}
f_0:=x_0 + x_1 + \dots + x_{n-1} + u_0 & = & 0\\
f_1:=x_0 x_1 + x_1 x_2 + \ldots + x_{n-1} x_0 + u_1 &=& 0 \\
\vdots \\
f_{n-2}:=x_0 x_1 \cdots x_{n-2} + \ldots + x_{n-1} x_0 \cdots x_{n-3} + u_{n-2} &=& 0\\
f_{n-1}:=x_0 x_1 \cdots x_{n-1} + u_{n-1} &= & 0 
\end{array}
\end{equation}

The standard cyclic n-roots problem is to  solve the system for a special choice of parameters $u_0=\ldots=u_{n-2}=0$ and $u_{n-1}=-1$. We will consider a variant of this problem where we solve the system for a general choice of parameters. 
For $n=5,6,7$ we  solve the parameterized system of equations, which can be deformed to the special choice of parameters and find all isolated nonsingular solutions. 
In the three cases we considered, the root count for the generic case agrees with the special case (this is no longer true for $n=4,8,9$ \cite{Haagecyclic}).

The system \eqref{eq:cyclicn} is known to have $70$ solutions when  $n=5$.
These solutions split into $7$ groups of $10$ elements via the dihedral action on the coordinates $x_0,x_1,x_2,x_3,x_4$ where rotations  act cyclicly on the labels and a reflection reverses the ordering of the labels. 
Defining equations for the irreducible curve $\cC$ are found  by restricting $u_0,u_1,\dots,u_4$ to a line parameterized by $v$.  

The projection $\pi:\cC\to\CC$ decomposes into $ \alpha \circ \beta$ where  $$
\alpha(v,x)=(v, x_3x_0+x_4x_1+x_0x_2+x_1x_3+x_2x_4)\, \text{ and } \, \beta(v,y)=v.
$$ To find this decomposition, we note that the system is invariant under the dihedral group $D_5$ which acts by label swapping the $x$-coordinates. 
We  use the Reynolds operator on the monomial $x_0x_2$, i.e.,
$$\frac{1}{|D_n|}\sum_{\sigma\in D_n}\sigma(x_0x_2)=x_{n-2}x_{0}+x_{n-1}x_{1}+x_0x_2+x_1x_3+\cdots+x_{n-3}x_{n-1}\quad(n=5).$$

One might be tempted to take $\alpha(v,x)=\left(v, \sum_{\sigma\in D_5}\sigma(x_0)\right)$ or $\alpha(v,x)=\left(v, \sum_{\sigma\in D_5}\sigma(x_0x_1)\right).$
However, such choices lead to $\alpha$ having degree $70$ as the entire fiber is mapped to a single point under $\alpha$; this means $\beta$ has degree $1$, and we fail to nontrivially decompose the projection. 

We summarize our computations for this subsection in Table \ref{table:cyclicroots}. The different homotopy continuation methods are as in Table~\ref{table:gaussian_mixture_computations}. We see again that Monodromy (degree) is much faster than total degree and polyhedral homotopies. 

\section*{Acknowledgements}
We would like to thank Bernd Sturmfels, Jonathan Hauenstein, Anton Leykin, Gunter Malle, and Botong Wang for their helpful comments and suggestions. 

\bibliographystyle{abbrv}
\bibliography{final_references.bib}

\begin{thebibliography}{10}

\bibitem{Alt}
H.~Alt.
\newblock {\"U}ber die {E}rzeugung gegebener ebener {K}urven mit {H}ilfe des
  {G}elenkvierecks.
\newblock {\em Zeitschrift f\"ur Angewandte Mathematik und Mechanik},
  3(1):13--19, 1923.

\bibitem{AFS}
C.~{Am\'endola}, J.-C. {Faug\`ere}, and B.~{Sturmfels}.
\newblock {Moment varieties of Gaussian mixtures}.
\newblock {\em Journal of Algebraic Statistics}, 7:14--28, 2016.

\bibitem{algident}
C.~Am{\'e}ndola, K.~Ranestad, and B.~Sturmfels.
\newblock Algebraic identifiability of gaussian mixtures.
\newblock {\em International mathematics research notices},
  2018(21):6556--6580, 2018.

\bibitem{Bertini}
D.~J. Bates, J.~D. Hauenstein, A.~J. Sommese, and C.~W. Wampler.
\newblock Bertini: Software for numerical algebraic geometry.
\newblock Available at \url{bertini.nd.edu} with permanent doi:
  dx.doi.org/10.7274/R0H41PB5, 2006.

\bibitem{homotopyjl}
P.~Breiding and S.~Timme.
\newblock Homotopycontinuation. jl: A package for homotopy continuation in
  julia.
\newblock In {\em International Congress on Mathematical Software}, pages
  458--465. Springer, 2018.

\bibitem{kileelchen}
J.~Chen and J.~Kileel.
\newblock Numerical implicitization.
\newblock {\em Journal of Software for Algebra and Geometry}, 9(1):55--63,
  2019.

\bibitem{DelCampo}
A.~M. del Campo and J.~I. Rodriguez.
\newblock Critical points via monodromy and local methods.
\newblock {\em Journal of Symbolic Computation}, 79(3):559 -- 574, 2017.
\newblock SI: Numerical Algebraic Geometry.

\bibitem{DK2002CIT}
H.~Derksen and G.~Kemper.
\newblock {\em Computational invariant theory}.
\newblock Encyclopaedia of mathematical sciences. Springer, Berlin, New York,
  2002.

\bibitem{polsyshommon}
T.~Duff, C.~Hill, A.~Jensen, K.~Lee, A.~Leykin, and J.~Sommars.
\newblock Solving polynomial systems via homotopy continuation and monodromy.
\newblock {\em IMA Journal of Numerical Analysis}, 39(3):1421--1446, 2019.

\bibitem{Haagecyclic}
U.~Haagerup.
\newblock Cyclic p-roots of prime lengths p and related complex {H}adamard
  matrices.
\newblock {\em arXiv:0803.2629}, 2008.

\bibitem{harris1979}
J.~Harris.
\newblock Galois groups of enumerative problems.
\newblock {\em Duke Math. J.}, 46(4):685--724, 12 1979.

\bibitem{tensorDecomp}
J.~D. Hauenstein, L.~Oeding, G.~Ottaviani, and A.~J. Sommese.
\newblock Homotopy techniques for tensor decomposition and perfect
  identifiability.
\newblock {\em Journal f{\"u}r die reine und angewandte Mathematik},
  2019(753):1--22, 2019.

\bibitem{HRmulti}
J.~D. Hauenstein and J.~I. Rodriguez.
\newblock Multiprojective witness sets and a trace test.
\newblock {\em Adv. Geom.}, 20(3):297--318, 2020.

\bibitem{galois}
J.~D. Hauenstein, J.~I. Rodriguez, and F.~Sottile.
\newblock Numerical computation of {G}alois groups.
\newblock {\em Foundations of Computational Mathematics}, Jun 2017.

\bibitem{Witness}
J.~D. Hauenstein and A.~J. Sommese.
\newblock Witness sets of projections.
\newblock {\em Applied Mathematics and Computation}, 217(7):3349--3354, 2010.

\bibitem{HSpseudo}
J.~D. Hauenstein and A.~J. Sommese.
\newblock Membership tests for images of algebraic sets by linear projections.
\newblock {\em Appl. Math. Comput.}, 219(12):6809--6818, Feb. 2013.

\bibitem{Deflation}
J.~D. Hauenstein and C.~W. Wampler.
\newblock Isosingular sets and deflation.
\newblock {\em Foundations of Computational Mathematics}, 13(3):371--403, 2013.

\bibitem{HomPS}
T.~Lee, T.~Li, and C.~Tsai.
\newblock Hom4ps-2.0: A software package for solving polynomial systems by the
  polyhedral homotopy continuation method, 2008.

\bibitem{tracetest}
A.~Leykin, J.~I. Rodriguez, and F.~Sottile.
\newblock Trace test.
\newblock {\em Arnold Math. J.}, 4(1):113--125, 2018.

\bibitem{LS09Galois}
A.~Leykin and F.~Sottile.
\newblock {Galois groups of Schubert problems via homotopy computation}.
\newblock {\em Math. Comput.}, 78(267):1749--1765, 2009.

\bibitem{LeykinDeflate}
A.~Leykin, J.~Verschelde, and A.~Zhao.
\newblock Newton's method with deflation for isolated singularities of
  polynomial systems.
\newblock {\em Theoretical Computer Science}, 359(1):111 -- 122, 2006.

\bibitem{dhagash}
D.~Mehta, Y.-H. He, and J.~D. Hauenstein.
\newblock Numerical algebraic geometry: a new perspective on gauge and string
  theories.
\newblock {\em Journal of High Energy Physics}, 2012(7):1--32, 2012.

\bibitem{powerbus}
D.~K. Molzahn, M.~Niemerg, D.~Mehta, and J.~D. Hauenstein.
\newblock Investigating the maximum number of real solutions to the power flow
  equations: Analysis of lossless four-bus systems.
\newblock {\em arXiv:1603.05908}, 2016.

\bibitem{Pirola}
G.~P. Pirola and E.~Schlesinger.
\newblock Monodromy of projective curves.
\newblock {\em J. Algebraic Geom.}, 14(4):623--642, 2005.

\bibitem{Monodromy}
A.~Sommese, J.~Verschelde, and C.~Wampler.
\newblock Using monodromy to decompose solution sets of polynomial systems into
  irreducible components.
\newblock In {\em Applications of algebraic geometry to coding theory, physics
  and computation ({E}ilat, 2001)}, volume~36 of {\em NATO Sci. Ser. II Math.
  Phys. Chem.}, pages 297--315. Kluwer Acad. Publ., Dordrecht, 2001.

\bibitem{TraceTesto}
A.~Sommese, J.~Verschelde, and C.~Wampler.
\newblock Symmetric functions applied to decomposing solution sets of
  polynomial systems.
\newblock {\em SIAM J. Numer. Anal.}, 40(6):2026--2046, 2002.

\bibitem{SWbook}
A.~J. Sommese and C.~W. Wampler, II.
\newblock {\em The numerical solution of systems of polynomials}.
\newblock World Scientific Publishing Co. Pte. Ltd., Hackensack, NJ, 2005.
\newblock Arising in engineering and science.

\bibitem{Sturmfels2008AIT}
B.~Sturmfels.
\newblock {\em Algorithms in Invariant Theory (Texts and Monographs in Symbolic
  Computation)}.
\newblock Springer Publishing Company, Incorporated, 2nd ed.; vii, 197 pp.; 5
  figs. edition, 2008.

\bibitem{PHC}
J.~Verschelde.
\newblock Phcpack: a general-purpose solver for polynomial systems by homotopy
  continuation.

\bibitem{NinePoint}
C.~Wampler, A.~Morgan, and A.~Sommese.
\newblock Complete solution of the nine-point path synthesis problem for
  four-bar linkages.
\newblock {\em ASME J. Mech. Design}, 114:153--159, 1992.

\bibitem{kinAlg}
C.~W. Wampler and A.~J. Sommese.
\newblock Numerical algebraic geometry and algebraic kinematics.
\newblock {\em Acta Numerica}, 20:469–567, May 2011.

\end{thebibliography}

\end{document}